\documentclass[12pt]{amsart}    
\usepackage{amscd}

\def\cal#1{\mathcal{#1}}

\def\NZQ{\Bbb}               
\def\NN{{\NZQ N}}

\def\ZZ{{\NZQ Z}}
\def\PP{{\NZQ P}}

\def\PP{{\NZQ P}}

\def\frk{\frak}               

\def\mm{{\frk m}}

\def\opn#1#2{\def#1{\operatorname{#2}}} 
\opn\chara{char}
\opn\length{\ell}
\opn\pd{pd}
\opn\rk{rk}
\opn\projdim{proj\,dim}
\opn\rank{rank}
\opn\depth{depth}
\opn\grade{grade}
\opn\height{ht}
\opn\embdim{emb\,dim}
\opn\codim{codim}
\def\OO{\mathcal{O}}
\opn\Tr{Tr}
\opn\bigrank{big\,rank}
\opn\superheight{superheight}\opn\lcm{lcm}
\opn\trdeg{tr\,deg}%
\opn\reg{reg}
\opn\lreg{lreg}
\opn\div{div}
\opn\Div{Div}
\opn\WDiv{WDiv}
\opn\cl{cl}
\opn\Cl{Cl}
\opn\Spec{Spec}
\opn\Supp{Supp}
\opn\supp{supp}
\opn\Sing{Sing}
\opn\Ass{Ass}
\opn\Assh{Assh}
\opn\Min{Min}
\opn\Reg{Reg}
\opn\Ann{Ann}
\opn\Rad{Rad}
\opn\Soc{Soc}
\opn\Socle{Socle}
\opn\Ker{Ker}
\opn\Coker{Coker}
\opn\Im{Im}
\opn\Hom{Hom}
\opn\Mor{Mor}
\opn\Tor{Tor}
\opn\Ext{Ext}
\opn\End{End}
\opn\Aut{Aut}
\opn\id{id}

\opn\nat{nat}
\opn\pff{pf}
\opn\Pf{Pf}
\opn\GL{GL}
\opn\SL{SL}
\opn\mod{mod}
\opn\ord{ord}
\opn\Proj{Proj}
\opn\aff{aff}
\opn\con{conv}
\opn\relint{relint}
\opn\st{st}
\opn\lk{lk}
\opn\cn{cn}
\opn\core{core}
\opn\vol{vol}
\opn\link{link}
\opn\star{star}
\opn\gr{gr}

\def\pot#1#2{#1[\kern-0.28ex[#2]\kern-0.28ex]}

\opn\dirlim{\underrightarrow{\lim}}
\opn\inivlim{\underleftarrow{\lim}}
\let\dirsum=\oplus
\let\tensor=\otimes
\let\iso=\cong
\let\Union=\bigcup

\let\Dirsum=\bigoplus

\let\to=\rightarrow

\let\To=\longrightarrow

\def\Implies{\ifmmode\Longrightarrow \else
     \unskip${}\Longrightarrow{}$\ignorespaces\fi}
\def\implies{\ifmmode\Rightarrow \else
     \unskip${}\Rightarrow{}$\ignorespaces\fi}
\def\iff{\ifmmode\Longleftrightarrow \else
     \unskip${}\Longleftrightarrow{}$\ignorespaces\fi}

\let\:=\colon
\opn\H{H}
\opn\Pic{Pic}

\newtheorem{Theorem}{Theorem}
\newtheorem{Corollary}[Theorem]{Corollary}

\newtheorem{Lemma}[Theorem]{Lemma}
\newtheorem{Proposition}[Theorem]{Proposition}
\newtheorem{Remark}{Remark}
\newtheorem{Example}[Remark]{Example}

\let\epsilon\varepsilon
\def\OO{{\cal O}} 

\opn\inii{in}
\opn\inim{inm}
\opn\set{set}
\def\pnt{{\raise0.5mm\hbox{\large\bf.}}}

\def\nondiv{\;{\not}\hspace*{-0.25em}\mid}

\begin{document}

\title{On non-vanishing of cohomologies of generalized Raynaud polarized surfaces}
\author{Yukihide Takayama}
\address{Yukihide Takayama, Department of Mathematical
Sciences, Ritsumeikan University, 
1-1-1 Nojihigashi, Kusatsu, Shiga 525-8577, Japan}
\email{takayama@se.ritsumei.ac.jp}

\def\Coh#1#2{H_{\mm}^{#1}(#2)}
\def\eCoh#1#2#3{H_{#1}^{#2}(#3)}

\newcommand{\AppTh}{Theorem~\ref{approxtheorem} }
\def\da{\downarrow}
\newcommand{\ua}{\uparrow}
\newcommand{\namedto}[1]{\buildrel\mbox{$#1$}\over\rightarrow}
\newcommand{\bdel}{\bar\partial}
\newcommand{\proj}{{\rm proj.}}

\newenvironment{myremark}[1]{{\bf Note:\ } \dotfill\\ \it{#1}}{\\ \dotfill
{\bf Note end.}}
\newcommand{\transdeg}[2]{{\rm trans. deg}_{#1}(#2)}
\newcommand{\mSpec}[1]{{\rm m\hbox{-}Spec}(#1)}

\newcommand{\tbf}{{{\Large To Be Filled!!}}}

\pagestyle{plain}
\maketitle

\begin{abstract}
We consider a family of slightly extended version of the Raynaud's
surfaces $X$ over the field of positive characteristic with
Mumford-Szpiro type polarizations ${\cal Z}$, which have Kodaira
non-vanishing $H^1(X, {\cal Z}^{-1})\ne{0}$.
The surfaces are at least normal but smooth under a
special condition. We compute the cohomologies 
$H^i(X, {\cal Z}^n)$ for $i, n\in\ZZ$ and study their
(non-)vanishing. Finally, 
we give a fairly large family of non Mumford-Szpiro type
polarizations ${\cal Z}_{a,b}$ with 
Kodaira non-vanishing.

MSC classification: (Primary: 14J25, 14J17; secondary: 13D99)
\end{abstract}

\def\gCoh#1#2#3{H_{#1}^{#2}\left(#3\right)}
\def\subsetneq{\raisebox{.6ex}{{\small $\; \underset{\ne}{\subset}\; $}}}
\opn\Exc{Exc}

\section{Introduction}

Let $X$ be a projective variety over an algebraically closed field $k$
and ${\cal Z}$ an ample invertible sheaf on $X$.
It is well known that Kodaira vanishing theorem does not hold 
if the characteristic of the field $\chara(k)=p$ is positive.
The first counter-example has been found by Raynaud \cite{Ray}.  
He constructed a smooth polarized surface $(X, {\cal Z})$
with  $H^1(X, {\cal Z}^{-1})\ne 0$ using Tango-structure \cite{Tan}.
Mukai \cite{Mu} generalized Raynaud's construction to obtain 
polarized
smooth projective varieties $(X, {\cal Z})$ of any dimension with 
$H^1(X, {\cal Z}^{-1})\ne{0}$. He also showed that, 
if a smooth projective surface $X$ is a counter-example to 
Kodaira vanishing, then $X$ must be
either hyperelliptic with $p=2,3$ or of general type.
The construction similar to  Mukai's has been also studied 
by Takeda \cite{Take3,Take4,Take07} and Russel \cite{Rus}.
Mumford \cite{M3} and Szpiro \cite{Szpiro} gave a sufficient condition
for a polarized smooth projective surface to be a counter-example to 
Kodaira vanishing and pointed out that Raynaud's examples are its instances.
Szpiro \cite{Szpiro2} and Lauritzen-Rao \cite{LR}
also gave  different counter-examples to Kodaira vanishing.
Mumford \cite{M1} constructed a normal polarized surface $(X, {\cal Z})$ with
$H^1(X, {\cal Z}^{-1})\ne{0}$ but it is not known whether
desingularizations of $X$ satisfy Kodaira vanishing.

The aim of this paper is to study (non-)vanishing of
$H^i(X, {\cal Z}^{n})$, $i, n\in\ZZ$, for 
a family of surfaces $X$ with Mumford-Szpiro type polarizations
${\cal Z}$, which is an extension of Raynaud's counter-examples.
Recall that Raynaud's examples are cyclic covers of ruled surfaces 
over smooth projective curves of genus $g \geq 2$.
The degree $\ell$ of the cyclic covers is $\ell=2$ 
for $p\geq 3$ and $\ell=3$ for $p=2$.
Notice that Kodaira vanishing holds for 
ruled surfaces \cite{Tan2, Mu}.
The smooth curve must have a special kind of divisor called
Tango-structure (Tango-Raynaud structure) and this gives a strong
restriction to the genus $g$ of the curve, i.e. $p$ must divide $2g -2$.
If we consider a weaker condition called pre-Tango, which is satisfied
by any smooth curves with $g\geq p$ (\cite{TakeYoko}), but in this case 
the obtained surface is singular.

As is implicitly described in \cite{Ray}, 
we can choose the degree $\ell$ of cyclic cover more freely.
In Mukai's construction \cite{Mu}, $\ell$ can be 
any integer $\geq 2$ with $(p,\ell)=1$ (and a mild condition),
but then we must take normalization to construct the cyclic cover.
In this paper, we consider an additional condition
$\ell\mid p+1$. 
This assures the normality of the cyclic cover
without normalization 
and moreover the computation of cohomologies $H^i(X, {\cal Z}^n)$ 
is much easier.
Thus, we obtain a fairly large class of surfaces over the fields of
positive characteristics containing many counter-examples to Kodaira
vanishing, together with formulas for cohomologies $H^i(X, {\cal Z}^n)$.
 These surfaces are normal if the base curve $C$ has 
a pre-Tango structure and smooth if $C$ has a Tango structure.

This family would be particularly interesting in the sense that
this provides a class of finitely generated graded integral
algebras $(R, \mm)$,
over the fields of positive characteristics whose graded local cohomologies
$\Coh{i}{R}$, $i<\dim{R}$, do not necessarily vanish at negative
degrees. 
For a polarized variety $(X, {\cal L})$, we consider the 
section ring $(R, \mm):= (\Dirsum_{n\geq 0}H^0(X, {\cal L}^n), 
\Dirsum_{n>0}H^0(X, {\cal L}^n))$, which is a finitely generated 
graded algebra over $k=H^0(X, \OO_X)$ with natural $\NN$-grading.
We have $X \iso \Proj(R)$. 
Then by computing $\check{{\rm C}}$ech complexes we know 
that $\Coh{0}{R}=0$ and we have 
\begin{equation*}
0\To R \To \Dirsum_{n\in\ZZ}H^0(X, {\cal L}^n) \To \Coh{1}{R}\To 0
\end{equation*}
and $\Coh{i+1}{R} \iso \Dirsum_{n\in\ZZ}H^i(X, {\cal L}^n)$ for 
$i\geq 1$.
From this, we immediately know that we always have 
$[\Coh{j}{R}]_n=0$ for $j=0,1$ and for all 
$n<0$ and moreover, if Kodaira type vanishing 
holds, then we have $[\Coh{i}{R}]_n=0$ for all $n<0$ and $i<\dim{R}$
(see \cite{HS}).
It is known that if $X$ has at most $F$-rational singularities and 
$X$ is obtained by generic mod $p$ reduction from a variety with at 
most rational singularities, then we have 
Kodaira type vanishing (see \cite{HW96, Hara, HS}).  
From our generalized Raynaud surfaces, we obtain 
examples of $R$ with $\dim R=3$ different from this type, 
whose local cohomologies  can be studied by analyzing cohomologies of 
vector bundles over smooth curves of  genus $\geq 2$.

In section~2, we will present the construction of our generalized
Raynaud surface $X$, which is the cyclic cover of degree $\ell$ of the
ruled surface $P$ over a curve $C$ with  pre-Tango structure.
In section~3, we show that $K_X$ is ample if $(p,\ell)=(3,4)$ and 
$p\geq 5$ (Proposition~\ref{prop:ampleKx}) and 
in this case we have Kodaira type vanishing 
$H^1(X, K_X^{-1})=0$ (Proposition~\ref{prop:KVforKx}). 
Then we apply the  Mumford-Szpiro type sufficient condition 
for Kodaira non-vanishing to obtain the polarization $(X,{\cal Z})$ 
with Kodaira non-vanishing (Proposition~\ref{prop:polarization}). 
Then we will compute cohomologies $H^i(X, {\cal Z}^n)$, $i,n\in\ZZ$,
in section~4 (Propositions~\ref{h2},~\ref{h1pos} and~\ref{h0},
Theorem~\ref{h1neg}, Corollary~\ref{h1neg-detailed}) 
and show some (non-)vanishing results 
(Corollaries~\ref{h2vanishing},~\ref{h1-for-p=2,3} and 
~\ref{h0-refined}, Theorem~\ref{result1}).
Finally, we give a class of polarizations with Kodaira non-vanishing, 
which are not of Mumford-Szpiro type (Theorem~\ref{prop:outkvce}).

The author thanks  Kei-ichi Watanabe, Masataka Tomari and 
Yushifumi Takeda for stimulating discussions.

\section{Fibered surfaces on pre-Tango curves}

In this section, we present the construction of our polarized surface,
which is a cyclic cover of a ruled surface over a smooth projective
curve. This is an extension of the Raynaud's counter-example
\cite{Ray} allowing more variations of the degree of the cyclic cover 
and a weaker condition for the base curve.
See \cite{Mu, Take3, Take4, Take07,TakeYoko, Xie} 
for similar constructions and detailed description.
In the following, let $k$ be an algebraically closed 
field of characteristic $\chara(k)=p>0$.

\subsection{pre-Tango and Tango structure}

Let $C$ be a smooth projective curve over $k$ with genus $g\geq 2$. 
We denote by $K(C)$ the function field of $C$ and we define 
$K(C)^p = \{ f^p\mid f\in K(C)\}$. Then the Tango-invariant
$n(C)$ is defined by 
\begin{equation*}
  n(C) := \max \left\{
           \deg\left[
		 \frac{(df)}{p}
                \right]
		 \mid 
               f\in K(C)\backslash K(C)^p
               \right\},
\end{equation*}
where $[\;\cdot\;]$ denotes round down of coefficients,
see \cite{Tan}.
We know that 
$0\leq n(C)\leq \displaystyle{\frac{2(g-1)}{p}}$
%
and  $C$ is called a {\em pre-Tango curve}
(or a {\em Tango curve}) if $n(C)>0$ (or 
$n(C) = \displaystyle{\frac{2(g-1)}{p}}$).
This means the existence of an ample divisor $D$ on $C$ 
such that $(df) \geq pD (>0)$ (or $(df) = pD (>0)$) with some 
$f\in K(C)\backslash K(C)^p$. We call 
the invertible sheaf ${\cal L}:= \OO_C(D)$ a {\em pre-Tango structure}
(or a {\em Tango structure}) of $C$.

Pre-Tango structure can be described in other way around. 
Consider the relative Frobenius morphism 
$F : C' \To C$ 
and let ${\cal B}^1$ be the image 
of the push forward $F_*d: F_*\OO_{C'}\To F_*\Omega^1_{C'}$
of the K\"{a}hler differential 
$d: \OO_{C'}\To \Omega_{C'}^1$.
Then we have the following short exact sequence
\begin{equation}
\label{B1exactseq}
0\To \OO_C \To F_*\OO_{C'} \To
 {\cal B}^1 \To 0.
\end{equation}
Now any ample invertible subsheaf ${\cal L}\subset {\cal B}^1$ 
is a pre-Tango structure of $C$ 
and the existence of such subsheaves is assured 
if $g\geq p$ (see Cor.~1.5 \cite{TakeYoko}), namely, 
curves with large genus are pre-Tango.

In the rest of this section, we 
consider a pre-Tango structure  ${\cal L} = \OO_C(D)$ of 
a pre-Tango curve $C$.

\subsection{dividing (pre-)Tango structure
\label{section:division-of-L}}

Consider the Jacobi variety ${\cal J}$ which consists of 
all the divisors of degree $0$ on $C$.
It is well known that if $(e, p)=1$, $e\in\NN$, 
the map $\varphi_e: {\cal J}\To {\cal J}$ s.t. $\varphi_e(D_0) = eD_0$
is surjective
(cf. page~42 \cite{M2}), i.e. every $D_0\in {\cal J}$
can be divided by $e$. 
Thus we know that,
for every $\NN \ni e\geq 2$ such that $(e,p)=1$ and $e\vert \deg {\cal
L}$,
there exists an ample invertible sheaf ${\cal N}$ with  ${\cal L } = {\cal N}^e$.

\subsection{Construction of the ruled surface $P$ and the divisor $E+C''$}

Tensoring (\ref{B1exactseq}) by ${\cal L}^{-1}$ to take  the global
sections, we have 
\begin{equation*}
0 \To H^0(C, {\cal B}^1\tensor{\cal L}^{-1})
  \overset{\eta}{\To}  H^1(C, {\cal L}^{-1})
  \overset{F^*}{\To} H^1(C, {\cal L}^{-p}).
\end{equation*}
On the other hand, we have the short exact sequence
\begin{equation}
\label{fundamentalExSeq2}
0\To {\cal B}^1 \To F_*\Omega_{C'}^1\overset{c}{\To}\Omega_{C}^1\To 0
\end{equation}
where $c$ is the Cartier operator \cite{Car}.
By tensoring (\ref{fundamentalExSeq2})
by ${\cal L}^{-1}$ to take the global sections,  we have 
\begin{equation*}
0 \To  H^0(C, {\cal B}^1\tensor {\cal L}^{-1})
  \To H^0(C, F_*(\Omega^1_{C'}(-p D)))
  \overset{c(-D)}{\To}   H^0(C, \Omega_{C}^1\tensor {\cal L}^{-1}).
\end{equation*}
Then we know $\Ker F^* \iso H^0(C, {\cal B}^1\tensor {\cal L}^{-1})
=\Ker{c(-D)}
\iso \{df \;\vert\; f\in K(C), (df)\geq pD\}$, which is non-trivial 
since $C$ is pre-Tango (cf. Lemma~12 \cite{Tan}).

Now take any $0\ne df_0\in H^0(C, {\cal B}^1\tensor {\cal L}^{-1})$.
Then  $\xi:= \eta(df_0)$ is a non-trivial
element in $H^1(C, {\cal L}^{-1})\iso \Ext^1_{\OO_C}({\cal
L}, \OO_C)$, so that we have a non-splitting extension
\begin{equation}
\label{fundamentalExSeq3}
 0 \To \OO_C \To {\cal E} \To {\cal L} \To 0
\end{equation}
where ${\cal E}$ is a locally free sheaf of rank $2$.

Moreover, we have $0=F^*\xi \in H^1(C, {\cal L}^{-p}) \iso
Ext^1_{\OO_C}({\cal L}^p, \OO_C)$ and the 
corresponding split extension
\begin{equation*}
 0 \To \OO_{C} \To F^*{\cal E} \To {\cal L}^{p}\To 0
\end{equation*}
is just the Frobenius pullback  of 
the sequence (\ref{fundamentalExSeq3}).
Using the splitting maps and tensoring by 
${\cal L}^{-1}$ we obtain another exact sequence
\begin{equation}
\label{fundamentalExSeq4}
 0 \To \OO_{C} \To F^*{\cal E}\tensor{\cal L}^{-p} 
   \To {\cal L}^{-p} \To 0.
\end{equation}

Now from the sequences (\ref{fundamentalExSeq3})
and (\ref{fundamentalExSeq4}) we obtain two ruled surfaces and 
their canonical cross sections $\sigma$ and $\tau$. Namely,
\begin{equation*}
   \pi : P = \PP({\cal E}) \To C,\quad E := \sigma(C)\subset P
\end{equation*}
where
$E$ is determined, as a Cartier divisor,
by the global section $s$ that is the 
image of $1$ by the inclusion 
$H^0(C, \OO_C) \hookrightarrow H^0(C, {\cal E}) = H^0(P, \OO_P(1))$
induced from (\ref{fundamentalExSeq3})
and 
\begin{equation*}
   \pi' : P' = \PP(F^*{\cal E}\tensor{\cal L}^{-p})
          \iso \PP(F^*{\cal E}) \To C\quad 
   \tilde{C}^{''} := \tau(C)
\end{equation*}
where 
$\tilde{C}^{''}$ is determined, as a Cartier divisor,
by the global section $t~{''}$ that is the 
image of $1$
by the inclusion $H^0(C, \OO_{C}) \hookrightarrow
H^0(C', F^*{\cal E}\tensor{\cal L}^{-p}) \iso H^0(P', \OO_{P'}(1))$
induced from (\ref{fundamentalExSeq4}).
Now we define the morphism $\varphi:P\To P'$ over $C$
by taking the $p$th power of the coordinates of 
$\pi^{-1}(x) \iso\PP^1_k ( \subset P)$ to obtain the coordinates of 
$(\pi')^{-1}(x) \iso\PP^1_k (\subset P')$ for every $x\in C$.
Then we set $C^{''} = \varphi^{-1}(\tilde{C}^{''})$.
By construction, 
we have $\OO_P(C^{''})\iso \OO_P(p)\tensor \pi^*{\cal L}^{-p}$ and
$C^{''}$ is a degree $p$ curve in $P$.
We know that $E\cap C^{''}=\emptyset$.
$E$ is smooth since $E \iso C$ via $\sigma$. 

\subsection{Purely inseparable cover  $\pi\mid_{C''}: C''\to C$}

Now as a Cartier divisor we write $D =\{(U_i, g_i)\}_i$, where 
$C = \Union_iU_i$ is an open covering, 
$g_i\in K(C)$ is 
the local equation of $D$.
By taking a finer covering, we can assume that ${\cal E}\mid_{U_i}$ 
are the free $\OO_{U_i}$-modules.
Then we can describe  $(df) \geq pD$ 
by $f = \{(U_i, g_i^pc_i)\}_i\in K(C)$ 
with $K(C)^p\not\ni c_i\in \OO_{U_i}$, i.e. $(df)\mid_{U_i} = (g_i^pdc_i)$. 
Then, 
\begin{Proposition}
\label{defideals}
We have 
\begin{equation*}
 C''\mid_{U_i} = \Proj \OO_{U_i}[x,y]/(c_ix^p+y^p).
\end{equation*}
In particular, $C''$ is an purely inseparable covering of $C$.
\end{Proposition}

\begin{proof}
The sequence (\ref{fundamentalExSeq3}) is locally  as follows:
\begin{equation*}
 0 \To \OO_{U_i} 
   \To \OO_{U_i}\dirsum \OO_{U_i}g_i^{-1}  
   \To \OO_{U_i}g_i^{-1} \To 0
\end{equation*}
so that we have 
$P\mid_{U_i} = \Proj (S(\OO_{U_i}\dirsum \OO_{U_i}g_i^{-1}))
=\Proj \OO_{U_i}[x,y]$,
where the indeterminates $x$ and $y$ represent the free basis
$1$ and $g_i^{-1}$.
On the other hand, 
we know that the sequence 
(\ref{fundamentalExSeq4}) is locally as follows:
\begin{equation*}
 0 \To \OO_{U_i} 
   \overset{i}{\To} (\OO_{U_i}\tensor \OO_{U_i} g_i^{-p})
                     \tensor\OO_{U_i}g_i^p
   \overset{j}{\To} \OO_{U_i} g_i^p
    \To 0
\end{equation*}
where we view 
\begin{eqnarray*}
(\OO_{U_i}\tensor \OO_{U_i}\cdot g_i^{-p})\tensor\OO_{U_i}\cdot g_i^p
&\iso& \OO_{U_i}\cdot g_i^p \dirsum \OO_{U_i}\cdot 1  \\
&\iso& \OO_{U_i}(c_ig_i^p,1) \dirsum \OO_{U_i}(g_i^p, 0)
\end{eqnarray*}
and we define  $i(a) = a(c_ig_i^p,1)$  and 
$j(a(c_ig_i^p,1) + b (g_i^p, 0)) = b g_i^p$
for $a, b\in\OO_{U_i}$.
Thus 
$P'\mid_{U_i}
=\Proj(S((\OO_{U_i}g_i^p\dirsum \OO_{U_i})))
\iso \Proj \OO_i[x',y']$ where the indeterminates $x'$ and $y'$
represents the free basis $g_i^p$ and $1$. 
Also $\tilde{C}''=\tau(C)$
is locally the zero locus of $t''\mid_{U_i} = c_ig_i^p+1$, so that 
we have 
\begin{equation*}
    \tilde{C}''\mid_{U_i} = \Proj \OO_{U_i}[x', y']/(c_ix' + y')
\end{equation*}
Since 
$\varphi: P=\Proj \OO_{U_i}[x,y]\To P' = \Proj \OO_{U_i}[x', y']$ is induced by 
the Frobenius 
$\OO_{U_i}[x',y']\ni x', y' \mapsto x^p, y^p \in \OO_{U_i}[x,y]$, we have 
\begin{equation*}
    C''\mid_{U_i} = \Proj \OO_{U_i}[x, y]/(c_ix^p + y^p).
\end{equation*}
\end{proof}

\begin{Remark}
\label{remark:idealsE}
{\em By a similar discussion to the proof of Proposition~\ref{defideals}
we can show 
\begin{equation*}
   E\mid_{U_i} = \Proj (\OO_{U_i}[x,y]/(x)) \iso \Spec \OO_{U_i}[y].
\end{equation*}
}
\end{Remark}

Later we will construct cyclic covers of $P$ ramified at $E + C''$ 
and the smoothness of the cyclic covers depends on the smoothness of 
$E$ and $C''$. Since $E \iso C$ is smooth by definition, we have to
see if $C''$ is smooth. To this end, we must prepare the following 
lemma.

\begin{Lemma}
\label{lemma:forSmoothnessofC''}
$\Omega_{C''/C}\iso \pi^*\OO_{C}(D)$.
\end{Lemma}
\begin{proof}By Proposition~\ref{defideals},
we have $\OO_{C''}\mid_{U_i} = \OO_{U_i}[x_i,y_i]/(c_ix_i^p + y_i^p)$
and $(df)\geq pD$ with $f = \{(U_i,g_i^pc_i)\}_i\in K(C)$.
Thus on  $U_i\cap U_j$ we have 
$g_i^pc_i = g_j^pc_j$ so that 
\begin{equation*}
(g_j^pc_j)(g_i^{-1}x_i)^p + y_i^p
=c_ix_i^p + y_i^p 
= c_jx_j^p + y_j^p 
= (g_j^pc_j)(g_j^{-1}x_j)^p + y_j^p
\end{equation*}
Thus we have $g_i^{-1}x_i=g_j^{-1}x_j$ and $y_i=y_j$, and then
\begin{equation*}
g_i^{-1}dx_i = d(g_i^{-1}x_i) =  d(g_j^{-1}x_j) = 
g_j^{-1}dx_j
\end{equation*}
for $d := d_{C''/C}$.
Now on $\tilde{U}_i= U_i\cap \{y_i\ne{0}\}$,
we have $\OO_{C''}\mid_{\tilde{U}_i}
= \OO_{\tilde{U}_i}[X_i]/(c_iX_i^p+1)$
with $X_i := x_i/y_i$ and then
\begin{equation*}
\Omega_{C''/C}\mid_{\tilde{U}_i}
= \OO_{C''}\mid_{\tilde{U}_i} \cdot g_i^{-1}dX_i
=\pi^*\OO_C(D)\mid_{\tilde{U}_i}
\end{equation*}
Notice that we have $g_i^{-1}dX_i = g_j^{-1}dX_j$
on $\tilde{U}_i\cap \tilde{U}_j$.

On the other hand, on $\hat{U}_i= U_i \cap \{x_i\ne{0}\}$ 
we can write 
\begin{equation*}
\OO_{C''}\mid_{\hat{U}_i} = \OO_C[Y_i]/(c_i + Y_i^p)
\qquad\mbox{with } Y_i = x_i^{-1}
\end{equation*}
and then
\begin{equation*}
\Omega_{C''/C}\mid_{\hat{U}_i} 
= \OO_{C''}\mid_{\hat{U}_i}\cdot dY_i
= \OO_{C''}\mid_{\hat{U}_i}\cdot x_i^{-2} dx_i.
\end{equation*}
By setting $t = g_i^{-1}x_i = g_j^{-1}x_j$ on $U_i\cap U_j$,
we have 
\begin{equation*}
x_i^{-2} dx_i = x_i^{-2} g_i dt = g_i^{-1}t^{-2}dt
\end{equation*}
and 
so that we have 
$\Omega_{C''/C}\mid_{\hat{U}_i}\iso \pi^*\OO_C(D)\mid_{\hat{U}_i}$.
Consequently, we have $\Omega_{C''/C} \iso \pi^*\OO_C(D)$,  
locally free of rank~1, as required.
\end{proof}

The following result first appeared in
Mukai's paper in Japanese (\cite{Mu} Prop.~5)
with a brief outline of the proof and his result 
is for varieties of arbitrary dimensions. 
We give here a detailed proof in the case of curves
for the readers convenience.

\begin{Theorem}
\label{theorem:smoothnessofC''}
Let $C$ be a pre-Tango curve. Then 
$C''$ is smooth if and only if $C$ is Tango.
\end{Theorem}
\begin{proof}
In the following 
we will denote the restriction of  $\pi:P\To C$ to $C''\subset P$
also by $\pi$. 
Now we consider the sequence
\begin{equation}
\label{basicsequence}
0\To \pi^*\OO_C(pD) 
 \overset{df}{\To}  \pi^*\Omega_C
 \overset{\psi}{\To} \Omega_{C''}
 \overset{\rho}{\To} \Omega_{C''/C} 
 \To 0
\end{equation}
where $df$ is  the multiplication by 
$df = \{(g_i^pdc_i)\}_i$.
The exactness of 
$\pi^*\Omega_C\To \Omega_{C''}\To\Omega_{C''/C}\To 0$
is well known.
The multiplication by $df$ 
is injective 
since $dc_i\ne{0}$ and $C$ is smooth. 
Moreover we have $\Ker\psi \supset \Im{df}$.  
To see this we have only to show that $\psi(dc_i)$,
which is by definition the K\"ahler differential of 
the image of $c_i$ by $\pi^\sharp :\OO_C\To \OO_{C''}$,
is trivial. But this is immediate since,
by Proposition~\ref{defideals},
$\pi^\sharp :\OO_C\To \OO_{C''}$
is locally the canonical inclusion
$\OO_C \hookrightarrow \OO_C[x,y]/(c_ix^p + y^p)$.
Thus (\ref{basicsequence}) is exact if and only if 
$\Ker\psi \subset \Im{df}$.

Now $\Omega_C$ is locally free of rank $1$ since $C$ is 
smooth. Then we know by NAK that 
$C$ being Tango, i.e. $(dc_i) = 0$
is equivalent with $\Coker(df)=0$.
This implies that (\ref{basicsequence}) is exact,
and 
then we have $\Omega_{C''} \iso  \Omega_{C''/C}$, 
which is locally free of rank~$1$ by Lemma~\ref{lemma:forSmoothnessofC''}, 
and $C''$ is smooth.
Conversely, assume that $C''$ is smooth. 
Then since $\Omega_{C''/C}$ and $\Omega_{C''}$ are 
locally free module of rank~$1$, we must have 
$\Ker\rho=\Im\psi =0$ so that we have $\Coker(df)=0$,
i.e., $C$ is Tango.
\end{proof}

\subsection{Construction of cyclic cover of $P$ ramified at $E+C''$
\label{section:cyclicCover}}
In this section, we will construct a cyclic cover $X$ of $P$ of suitable
degree ramified at $E+C''$.
We choose $\ell\geq 2$ such that $\ell\mid p+1$ and $\ell\mid e$, 
and set 
\begin{equation*}
{\cal M}:= \OO_P\left(
		 -\frac{p+1}{\ell}
		\right)\tensor \pi^*{\cal N}^{\frac{pe}{\ell}}.
\end{equation*}
Then we have ${\cal M}^{-\ell} = \OO_P(E+C'')$.
Now we define an $\OO_P$-algebra structure in
$\Dirsum_{i=0}^{\ell-1}{\cal M}^i$  with the multiplication
defined by 
\begin{equation*}
\begin{array}{ccc}
{\cal M}^i \times {\cal M}^j &\To&  {\cal M}^{i+j} \\
   (a, b)              & \mapsto &  a\tensor b\\
\end{array}
\end{equation*}
if $i+j\leq \ell-1$ and 
\begin{equation*}
\begin{array}{ccc}
{\cal M}^i \times {\cal M}^j &\To& {\cal M}^{i+j-\ell} \\
   (a, b) & \mapsto& a\tensor  b\tensor \zeta\\
\end{array}
\end{equation*}
if $i+j> \ell$, with $\zeta=s\tensor{t''}$ 
where $s$ and $t''$ are the global sections 
defining $E$ and $C''$.
Then we obtain 
\begin{equation*}
\psi : X := {\cal Spec} \left(\Dirsum_{i=0}^{\ell-1}{\cal M}^i\right)
\To P,
\end{equation*}
which is the cyclic cover of the ruled surface $P$
ramified at $E+C''$  of degree $\ell$,
where ${\cal Spec}$ denotes the affine morphism.
Now we will define $\phi = \pi\circ\psi : X\to C$.

\begin{Remark}
{\em 
$\phi: X\to C$
is an extension of Raynaud's original counter-example to
Kodaira vanishing. Namely, let $C$ be a Tango curve
and let $e=\ell=3$ if $p=2$ and  $e=\ell=2$ if $p\geq 3$,
then we obtain the example as given in \cite{Ray}.
}
\end{Remark}

We define $\tilde{E} = \psi^{-1}(E)$ and $\tilde{C}'' = \psi^{-1}(C'')$,
then  we have 
\begin{equation}
\label{liftByPsi}
 \ell\tilde{E} = \psi^*E\quad\mbox{and} \quad
 \ell\tilde{C}'' = \psi^*C''
\end{equation}
and 
\begin{equation}
\label{psiPushForward}
  \psi_*\OO_X = \Dirsum_{i=0}^{\ell-1}{\cal M}^i.
\end{equation}
Moreover, we have the following, which will be used 
later.

\begin{Lemma} \label{tool2}
For $k\geq 1$, we have 
$\psi_*\OO_X(-k\tilde{E}) \iso \OO_P(-kE)\dirsum
\displaystyle{\Dirsum_{i=1}^{\ell-1}}{\cal M}^i$.
\end{Lemma}
\begin{proof}From the exact sequence
\begin{equation*}
0\To \OO_X(-k\tilde{E}) \To \OO_X \To \OO_{k\tilde{E}} \To 0,
\end{equation*}
we obtain by (\ref{psiPushForward}) 
\begin{equation*}
0\To \psi_*\OO_X(-k\tilde{E}) 
 \To \Dirsum_{i=0}^{\ell-1}{\cal M}^i
 \To \psi_*\OO_{k\tilde{E}}
 \To R^1\psi_*\OO_X(-k\tilde{E}),
\end{equation*}
where $R^1\psi_*\OO_X(-\tilde{E})=0$
since $\psi: X\to P$ is an affine morphism.
Also since $\psi: \tilde{E}\iso E$,
we have $\psi_*\OO_{k\tilde{E}}  \iso  \OO_{kE}=\OO_P/\OO_P(-kE)$.
Then we obtain the following diagram:
\begin{equation*}
\begin{array}{ccccccccc}
0 & \To & \psi_*\OO_X(-k\tilde{E}) & \To & \OO_P\dirsum \Dirsum_{i=1}^{\ell-1}{\cal
 M}^i
 & \To & \psi_*\OO_{k\tilde{E}} & \To & 0 \\
|| &  &     &  & ||   & & ||  & & || \\
0 & \To & \OO_P(-kE)\dirsum \Dirsum_{i=1}^{\ell-1}{\cal M}^i
  & \To & \OO_P\dirsum \Dirsum_{i=1}^{\ell-1}{\cal M}^i
  & \To & \OO_P/\OO_P(-kE) & \To & 0 \\
\end{array}
\end{equation*}
from which we have $\psi_*\OO_X(-k\tilde{E})\iso \OO_P(-kE)\dirsum 
\Dirsum_{i=1}^{\ell-1}{\cal M}^i$ by 5-lemma.
\end{proof}

\begin{Lemma}
\label{tool1}
For $k\geq 0$ and $1\leq r\leq \ell-1$, we have 
\begin{enumerate}
\item [$(i)$] $\psi_*\OO_X(k\ell\tilde{E})
                =  \displaystyle{\Dirsum_{i=0}^{\ell-1}}{\cal M}^i(kE)$, 
\item [$(ii)$] $\psi_*(\OO_X((k\ell + r)\tilde{E}))
                =  \OO_P(r+1+k-\ell) 
                   \dirsum 
                   \displaystyle{\Dirsum_{i=1}^{\ell-1}}{\cal M}^i((k+1)E)$.
\end{enumerate}
\end{Lemma}
\begin{proof}
Using  (\ref{liftByPsi}) and  (\ref{psiPushForward}), 
we have 
$\psi_*\OO_X(k\ell \tilde{E})
 =  \psi_*\psi^*\OO_P(k E)
 =  \psi_*\OO_X\tensor\OO_P(kE)
 =  \Dirsum_{i=0}^{\ell-1}{\cal M}^i(kE)$,
which proves $(i)$. 
Now
for $1\leq r\leq \ell-1$, 
we compute by  (\ref{liftByPsi}) and  Lemma~\ref{tool2}
\begin{eqnarray*}
\lefteqn{\psi_*\OO_X(r\tilde{E})}\\
&= & \psi_*(\OO_X( -(\ell-r)\tilde{E}
                \tensor
                \ell\tilde{E})
             )
=  \psi_*(  \OO_X(-(\ell-r)\tilde{E})
           \tensor
           \psi^*\OO_P(E)
        )\\
&= &\psi_*\OO_X(-(\ell-r)\tilde{E})\tensor \OO_P(E)
= \OO_P((r+1-\ell)E) \dirsum \Dirsum_{i=1}^{\ell-1}{\cal M}^i(E),
\end{eqnarray*}
from which we immediately obtain $(ii)$.
\end{proof}

Our surface $X$ is at least normal even if $E+C''$ is singular,
namely in the case that $C$ is pre-Tango but not Tango
(see Theorem~\ref{theorem:smoothnessofC''}).
To prove this fact we 
use a result by Esnault-Viehweg. 
Let $Y$ be a variety, ${\cal H}$ an invertible sheaf over $Y$
and $E = \sum_{j=1}^r\alpha_jE_j$ an effective divisor
such that ${\cal H}^\ell = \OO_X(E)$ for some integer $\ell\geq 2$.
We define
${\cal H}^{(i)}:= {\cal H}^i\tensor\OO_X(-[\frac{i}{\ell}E])$
and set 
${\cal A} := \Dirsum_{i=0}^{\ell-1}{{\cal H}^{(i)}}^{-1}$.
Then we have 
\begin{Proposition}[cf. Claim~3.12 of \cite{EV}]
\label{EVprop}
The canonical morphism ${\cal Spec(A)}\To Y$ is finite
and ${\cal Spec(A)}$ is normal.
\end{Proposition}

\begin{Corollary}
\label{normality}
For every pre-Tango curve $C$, the above constructed 
surface $X$ is normal. In particular, $X$ is Cohen-Macaulay.
Moreover, $X$ is smooth if $C$ is a Tango curve.
\end{Corollary}
\begin{proof} We apply Prop.~\ref{EVprop} in the case of
$(X, {\cal H}, \ell, E) := (P, {\cal M}^{-1},\ell, E+C'')$.
Since $C''$ and $E$ are reduced curves, we have 
$\OO_P([\frac{i}{\ell}(E+C'')])  
=\OO_P$ for $0\leq i\leq \ell-1$. 
Thus  ${\cal Spec(A)}$ is nothing but our surface $X$
and normal.
Since $\dim X=2$, we know that $X$ is Cohen-Macaulay 
by Serre's $(S_2)$ condition.
The final statement follows from
Theorem~\ref{theorem:smoothnessofC''}.
\end{proof}

\begin{Remark}
\label{section:mukaisconstruction}
{\em
The cyclic cover constructed by Mukai \cite{Mu} is more general
than ours.
Let ${\cal L}, D, {\cal N}, e$ be as in 
\ref{section:cyclicCover}. Choose $\ell\geq 2$ such that 
$\ell\mid e$ and $(\ell,p)=1$. Notice that the last condition is 
weaker than our condition $\ell\mid (p+1)$.
Mukai's construction is as follows. 
For any $\alpha\in\NN$ such 
that $\ell \mid (p+\alpha)$ we write
\begin{equation*}
0\sim C'' - p E + p\pi^*(D) 
= C'' + \alpha E 
      + \ell K
\qquad 
\mbox{with}\quad
K :=  -\frac{p+\alpha}{\ell} E + \frac{p}{\ell} \pi^*D
\end{equation*}
and set 
\begin{equation*}
{\cal M}_\alpha := \OO_P(K) = 
\OO_P\left(-\frac{p+\alpha}{\ell}\right)
 \tensor \pi^*{\cal N}^{\frac{pe}{\ell}}.
\end{equation*}
Then, we have 
${\cal M}_\alpha^{-\ell} = \OO_P(C'' + \alpha E)$.
Now we consider 
\begin{equation*}
\psi': X' := {\cal Spec} \Dirsum_{i=0}^{\ell-1}{\cal M}_\alpha^i
\To P,
\end{equation*}
which is normal if and only if $\alpha=1$.
Thus we take the normalization $X$ of $X'$ to obtain 
the cyclic cover $\phi : X \to P$. 
Corollary~\ref{normality} also holds for
this construction.
But the normalization $f:X\to X'$ makes it difficult to
compute the cohomologies $H^i(X, {\cal Z}^n)
= H^i(X',f_*{\cal Z}^n)$ for a
polarization $(X,{\cal Z})$.
}
\end{Remark}

\section{Basic properties of the surfaces}

We will show  some basic properties of our surface $X$.
Also  we will define the Mumford-Szpiro type polarization 
$(X, {\cal Z})$ in the end of this section.

The cross section $\tilde{E} \subset X$ has a positive 
self-intersection number. 
\begin{Proposition}
\label{prop:E*E} 
The self-intersection number of $\tilde{E}$ is 
$\tilde{E}^2 = \displaystyle{\frac{1}{\ell}}\cdot \deg{D}\; (>0)$.
\end{Proposition}
\begin{proof}
We compute $\tilde{E}^2 
= \left(
\displaystyle{\frac{\psi^*(E)}{\ell}},
\displaystyle{\frac{\psi^*(E)}{\ell}}
\right)
= \displaystyle{\frac{\deg\psi}{\ell^2}}\cdot {E}^2
= \displaystyle{\frac{1}{\ell}}\cdot {E}^2
= \displaystyle{\frac{1}{\ell}}\cdot \deg D$.
\end{proof}

Now we consider the canonical divisor $K_X$.

\begin{Proposition}
\label{prop:Kx}
$K_X \sim \phi^*\left(K_C 
- \displaystyle{\frac{p\ell - p - \ell}{\ell}}\cdot D\right)
 + (p\ell - p - \ell -1)\tilde{E}$.
\end{Proposition}
\begin{proof}
We have 
$K_X \sim  \psi^*K_P + (\ell -1)\tilde{E} + (\ell -1)\tilde{C}''$
by the branch formula. Applying $\psi^*$ to $C^{''} \sim p E - p \pi^*D$
to obtain 
$\tilde{C}'' \sim p \tilde{E} - \frac{p}{\ell}\cdot \phi^*(D)$.
Then a direct computation, together with the 
well known formula $K_P = -2 E + \pi^*K_C + \pi^*(D)$,
shows the required result.
\end{proof}

\begin{Proposition}
\label{prop:ampleKx}
$K_X$ is ample if $(p,\ell) = (3,4)$ or $p\geq 5$.
\end{Proposition}
\begin{proof}
We have $K_X = \phi^*A + B$, where $A = K_C - (p\ell - p - \ell)D/\ell$,
$B = (p\ell - p - \ell -1)\tilde{E}$
by Proposition~\ref{prop:Kx}.
Since $0 < \deg D \leq \frac{2(g-1)}{p}$ and $g\geq 2$, 
we see $\deg A >0$.
Also we see that 
$\deg B \leq 0$ if and only if $(p,\ell) = (2,2), (2,3), (3,2)$.
Thus, since $\ell\mid (p+1)$, we have 
$\deg B >0$ if and only if $(p,\ell) = (3, 4)$ and $p\geq 5$.
In these cases, we have $K_X^2>0$ and $K_X.H>0$ for every irreducible 
curve $H\in\Pic(P)$ in $P$ by Proposition~\ref{prop:E*E}
(cf. Prop.~V.2.3 \cite{H}). Thus $K_X$ is ample by 
Nakai-Moishezon criterion.
\end{proof}

Now we are interested in whether $H^1(X, K_X^{-1})=0$ holds 
if $K_X$ is ample.

\begin{Lemma}We have 
\begin{equation*}
H^1(X, K_X^{-1})
=
\Dirsum_{i=0}^{\ell-1}
H^1(P,         K_P^{-1}
               - \frac{(p+1)(\ell-1+i)}{\ell}E
               + \frac{p(\ell-1+i)}{\ell}\pi^*D
           )
\end{equation*}
\end{Lemma}
\begin{proof}
Since $\psi: X\to P$ is an affine morphism, we have $H^1(X, K_X^{-1}) =
H^1(P, \psi_*(K_X^{-1}))$.
By branch formula and $C''\sim p E - p \pi^*D$, we have 
\begin{equation*}
K_X^{-1}
=   \psi^*\left(
             K_P^{-1} - \frac{(p+1)(\ell-1)}{\ell}E
             + \frac{p(\ell-1)}{\ell} \pi^*D
           \right)
\end{equation*}
so that 
\begin{eqnarray*}
\psi_*(\OO_X(K_X^{-1}))
& = & \psi_*\OO_X
      \tensor_{\OO_X}
       \OO_P\left(
             K_P^{-1} - \frac{(p+1)(\ell-1)}{\ell}E
             + \frac{p(\ell-1)}{\ell} \pi^*D
           \right)
\end{eqnarray*}
where 
$\psi_*\OO_X 
    = \Dirsum_{i=0}^{\ell-1}{\cal M}^i$.
Then we obtain the above stated result.
\end{proof}

Since  Kodaira vanishing holds for $P$
(see \cite{Mu, Tan2}), to show that 
$H^1(X, K_X^{-1})=0$ we have only to show that 
\begin{equation*}
L_i:= K_P +  \frac{(p+1)(\ell-1+i)}{\ell}E
               - \frac{p(\ell-1+i)}{\ell}\pi^*D
\end{equation*}
are ample for $i=0,\ldots, \ell-1$. 

\begin{Proposition}
\label{prop:KVforKx}
$H^1(X, K_X^{-1})=0$ holds for $p\geq 5$ or $p=3$ and  $\ell=e=4$.
\end{Proposition}
\begin{proof}
Let ${\bf f}$ be  any fiber of $\pi:P \to C$. Then 
we have $\pi^*D= \deg D\cdot {\bf f}$
and the numerical equivalence 
$K_P \equiv -2 E + 2(g-1)\cdot {\bf f} + \deg D \cdot {\bf f}$.
Thus we have  $L_i \equiv u_i\cdot E + v_i \cdot {\bf f}$
where 
\begin{equation*}
u_i :=  \frac{(p+1)(\ell-1+i)}{\ell}-2,\quad
v_i := 2g-2 - \frac{p\ell-p-\ell + pi}{\ell}\cdot\deg D.
\end{equation*}
Then, using the condition $\ell\mid (p+1)$, we can show that 
$L_i.E>0$ and $L_i.{\bf f}>0$ if $p\geq 5$ or $p=3$ and $\ell=e=4$.
Also a straightforward computation shows that
$L_i^2 >0$. Then by Nakai-Moishezon's criteria, $L_i$,
$i=0,\ldots, \ell-1$, are  ample.
\end{proof}

Next we consider the fibers $X_y := \phi^{-1}(y)\; (\subset X)$
for $y\in C$.

\begin{Proposition}
\label{prop:cusp}
Every $X_y$ has a singularity
at  the intersection with the curve $\tilde{C}''$,
which is the cusp of the form $Z^\ell = W^p$.
\end{Proposition}
\begin{proof}
The fiber $X_y$ may have singularities 
at the intersection with $\tilde{E} + \tilde{C}''$, which 
are the inverse image of $\PP^1 \cap (E + C'')\; (\subset P)$ 
by $\psi$.
By Proposition~\ref{defideals}, $C''\subset P$ is locally 
defined by the equations
$c_i X^p + Y^p\in \OO_{U_i}[X,Y]$ with $y\in U_i$.
Thus
$Z = (c_iX^p+ Y^p)^{1/\ell}$ is a local coordinate of 
$\phi^{-1}(U_i)\subset X$.
Setting the new coordinate 
$W = c_i^{1/p}X + Y$ we have $Z^\ell = W^p$ as required.
Moreover, a similar argument shows that $\phi^{-1}\cap E$
is not singular (cf. Remark~\ref{remark:idealsE}).
\end{proof}

Although $X_y$ is birational with $\PP^1$, it has a positive 
geometric genus.

\begin{Proposition}
\label{prop:genusOfFiber}
The geometric genus of $X_y$ is 
$\displaystyle{\frac{(\ell-1)(p-1)}{2}}\;(>0)$.
\end{Proposition}
\begin{proof}
By normalization, we can assume that $X_y$ is smooth and 
 $\psi_y=\psi\mid_{X_y}: \phi^{-1}(y) \To  \PP^1\iso\pi^{-1}(y)$,
$y\in C$, 
is a finite separated morphism. 
By taking the normalization, we can assume from the 
beginning that $X_y:= \phi^{-1}(y)$ is a smooth curve and 
the degree $\deg \psi_y(=\ell)$ is preserved.
Also the ramification divisor for $\psi_y$ 
is $(\ell-1)(\tilde{E}\cap X_y) + (\ell-1)(\tilde{C}''\cap X_y)$,
whose degree is $(\ell -1)(p+1)$.
Thus by Hurwitz formula  we obtain the required result.
\end{proof}

Mumford and Szpiro generalized 
Raynaud's examples and obtained the following result.
\begin{Theorem}[Mumford and Szpiro \cite{M3,Szpiro}]
\label{thm:mumford-szpiro}
Let $\phi : X\to C$ be a fibration from a smooth projective surface
to a smooth projective curve and assume that 
each fiber is reduced and irreducible with positive geometric
genus. Then if there exists a cross section $\Gamma \subset X$
of $\phi$ with positive self intersection number, we have 
$(i)$ ${\cal Z} = \OO_X(\Gamma)\tensor 
\phi^*(\phi_*\OO_X(\Gamma)\vert_\Gamma)$ 
is ample, and 
$(ii)$ $H^1(X, {\cal Z}^{-1})\ne 0$. 
\end{Theorem}

By Proposition~\ref{prop:E*E} and
Proposition~\ref{prop:genusOfFiber}, 
we know that our surface $X$ is an instance of 
this theorem when  $\Gamma = \tilde{E}$. Moreover, 

\begin{Proposition}
\label{prop:polarization}
In this case, we have 
${\cal Z} = 
\OO_X(\tilde{E})\tensor\phi^*{\cal N}^{e/\ell}
= \OO_X(\tilde{D})$
where $\tilde{D} = \psi^{-1}(E) + \phi^{-1} D'
\mbox{ with } D' = \frac{1}{\ell}D$.
\end{Proposition}
\begin{proof}
We have 
$\phi_*\OO_X(\tilde{E})\mid_{\tilde{E}}
= \phi_*(\OO_X(\tilde{E})\tensor\OO_{\tilde{E}}) 
= (\pi_*\circ \psi_*) (\psi^*\OO_P(\frac{1}{\ell}E)\tensor\OO_{\tilde{E}}) 
= \pi_*(\psi_*\OO_{\tilde{E}}\tensor \OO_P(\frac{1}{\ell}E))$.
Now from the short exact sequence
\begin{equation*}
  0 \To \OO_X(-\tilde{E}) \To \OO_X \To \OO_{\tilde{E}} \To 0
\end{equation*}
we obtain, 
\begin{equation*}
  0 \To \psi_*\OO_X(-\tilde{E}) \To \psi_*\OO_X \To 
\psi_*\OO_{\tilde{E}} \To R^1\psi_*\OO_X(-\tilde{E})
\end{equation*}
and $R^1\psi_*\OO_X(-\tilde{E})=0$ since $\psi$ is an affine morphism.
Thus 
by Lemma~\ref{tool2} we have 
\begin{equation*}
\psi_*\OO_{\tilde{E}}
\iso \psi_*\OO_X/\psi_*\OO_X(-\tilde{E})
\iso \frac{\Dirsum_{i=0}^{\ell-1}{\cal M}^i}
          {\OO_P(-E)\dirsum \Dirsum_{i=1}^{\ell-1}{\cal M}^i}
\iso \OO_E,
\end{equation*}
and then
\begin{equation*}
\phi_*\OO_X(\tilde{E})\mid_{\tilde{E}}
= \pi_*(\OO_{E}\tensor \OO_P(\frac{1}{\ell}E))
= \OO_C(\frac{1}{\ell}D) (= {\cal N}^{e/\ell}).
\end{equation*}
since $E$ is the canonical section of $\pi:P\to C$.
\end{proof}

Notice that if $e=\ell=2$ when $\chara{k}\geq 3$ and $e=\ell=3$
when $\chara{k}=2$ then ${\cal Z}$ in Proposition~\ref{prop:polarization}
is the same as the ample invertible sheaf of Raynaud's
counter-example to Kodaira vanishing.

\section{Cohomologies for Mumford-Szpiro type polarization}
\def\N{{\cal N}_\ell}

In this section, we will compute the cohomologies 
$H^i(X, {\cal Z}^n)$ for the Mumford-Szpiro type polarization $(X, {\cal
Z})$ given in Proposition~\ref{prop:polarization}, namely
${\cal Z}= \OO_X(\tilde{E})\tensor\phi^*\N$
where we set $\N = {\cal N}^{\frac{e}{\ell}}$.

First of all, we summarize the well-known facts about ruled surfaces,
which are necessary in our computation of cohomologies.

\begin{Lemma} 
\label{prop:basicsRuledSurf}
For the ruled surface $\pi:P = \PP({\cal E})\To C$, we have 
\begin{enumerate}
\item [$(i)$] 
$\pi_*\OO_P(k) = S^k({\cal E})$, which is the $k$th component
of the symmetric algebra $S({\cal E})$. We will understand 
$S^k({\cal E})=0$ for $k<0$.
\item [$(ii)$]
for a locally free sheaf ${\cal F}$ 
on $P$,
$H^i(P, \OO_P(n)\tensor \pi^*{\cal F})\iso 
H^i(C, S^n({\cal E})\tensor {\cal F})$
for $n\geq 0$ and $i\in\ZZ$.
\item [$(iii)$] 
\begin{equation*}
R^1\pi_*\OO_P(n)
 = \left\{
     \begin{array}{ll}
        0 & \mbox{if } n\geq -1   \\  
        S^{-n-2}({\cal E})^\vee\tensor{\cal L}^\vee
          & \mbox{if } n\leq -2   \\  
     \end{array}
   \right.
\end{equation*}
\end{enumerate}
\end{Lemma}
For an extension of $(ii)$ with ${\cal F}$ any coherent
sheaf, see Proposition~7.10.13~\cite{Ando}.
\begin{proof}
$(i)$ is Proposition~II.7.11(a)~\cite{H}. 
Now we have 
$R^1\pi_*\OO_P(n)
= \pi_*\OO_P(-(n+2))^\vee \tensor {\cal L}^\vee$
by Exer.~III.8.4(c)~\cite{H}. Then
applying $(i)$ we obtain $(iii)$, cf. Appendix~A~\cite{Laz1}.
Finally, we have $S^n({\cal E})\tensor{\cal F}
\iso \pi_*(\OO_P(n)\tensor\pi^*{\cal F})$ by $(i)$
and 
$R^i\pi_*(\OO_P(n)\tensor\pi^*{\cal F})=
R^i \pi_*\OO_P(n) \tensor {\cal F} = 0$ for $n\geq 0$ and $i\geq 0$
by $(iii)$. Thus we obtain $(ii)$ by Leray spectral sequence.
\end{proof}

The following vanishing result will also be used.

\begin{Proposition}
\label{prop:vanishingONcurves} 
For any $1\leq m$ and $k<e$, we have 
$H^0(C, S^m({\cal E})^\vee\tensor {\cal N}^k)=0$.
\end{Proposition}
\begin{proof}
Since $\rank{\cal E}=2$ and  ${\cal L}$ is the 
surjective image of ${\cal E}$ ,
we have $\rank S^m({\cal E})=m+1$ and 
there exists a  filtration
\begin{equation*}
     0= {\cal F}_0 
\subset {\cal F}_1
\subset \cdots 
\subset {\cal F}_m
\subset {\cal F}_{m+1} = S^m({\cal E})
\end{equation*}
such that ${\cal F}_j$ is a locally free sheaf of 
$\rank {\cal F}_j=j$ and 
\begin{equation*}
0\To {\cal F}_{j-1} \To {\cal F}_{j} \To  {\cal L}^j\To 0
\end{equation*}
for $j=1,\ldots, m+1$.  
%
Now taking the dual, tensoring by ${\cal N}^{k}$ and 
taking the global sections, we have for $j=2,\ldots, m+1$
\begin{equation*}
0 \To H^0(C,{\cal L}^{-j}\tensor {\cal N}^{k})
  \To H^0(C,{\cal F}_{j}^\vee\tensor {\cal N}^{k})
  \overset{\psi_j}{\To} H^0(C,{\cal F}_{j-1}^\vee\tensor {\cal N}^{k}).
\end{equation*}
If $\deg {\cal L}^{j}\tensor {\cal N}^{-k}
= (ej -k)\cdot \deg {\cal N}>0$, i.e.  $ej>k$, 
we have $H^0(C,{\cal L}^{-j}\tensor {\cal N}^{k})=0$
so that $\psi_j$ is an inclusion.
Thus we have 
\begin{equation*}
\psi_2\circ\cdots\circ\psi_{m+1}:
  H^0(C, S^m({\cal E})^\vee\tensor {\cal N}^k)
\subset 
   H^0(C, {\cal F}_1^\vee\tensor {\cal N}^k)
=  H^0(C, {\cal L}^{-1}\tensor {\cal N}^k)
\end{equation*}
if $ej>k$ for all $j=2,\ldots, m+1$, i.e.
if $2e>k$. Now $H^0(C, {\cal L}^{-1}\tensor {\cal N}^k)=0$ 
and thus $H^0(C, S^m({\cal E})^\vee\tensor {\cal N}^k)=0$,
if $\deg {\cal L}^{-1}\tensor {\cal N}^k
= (k-e)\cdot \deg {\cal N}<0$, i.e. if $e>k$.
\end{proof}

\subsection{Computation of $H^2(X, {\cal Z}^n)$}

Now we compute $H^2(X, {\cal Z}^n)$.

\begin{Proposition}
\label{h2}
For $k\geq 0$ and  $1\leq r\leq \ell-1$,
\begin{eqnarray*}
\lefteqn{H^2(X, {\cal Z}^n)}\\
& =& 
\left\{
\begin{array}{ll}
\displaystyle{\Dirsum_{i=\left[\frac{n}{p+1}+1\right]}^{\ell -1}}
    H^2(P, 
       \OO_P\left(
               -\frac{i(p+1)}{\ell} + k
             \right)
        \tensor \pi^*{\N}^{ip + n}
      )
     & \mbox{if } n = k\ell \geq 0\\
 & \\
H^2(P, \OO_P(r+1+k-\ell))\tensor \pi^*{\N}^n)  &  \\
\quad\dirsum\;
\displaystyle{\Dirsum_{i=\left[\frac{n+\ell-r}{p+1}+1\right]}^{\ell -1}}
   H^2(P, 
       \OO_P\left(
                 -\frac{i(p+1)}{\ell} + k + 1
           \right)
       \tensor \pi^*{\N}^{ip + n}
       )
     & \mbox{if } n = k\ell + r > 0\\
 & \\
H^2(P, \OO_P(n)\tensor f^*{\N}^n) & \\
\quad\dirsum\; 
\displaystyle{\Dirsum_{i=1}^{\ell-1}}
H^2(P, 
\displaystyle{\OO_P\left(
             -\frac{i(p+1)}{\ell}
         \right)
     \tensor \pi^*{\N}^{ip+n}}
    )
     & \mbox{if } n <0. \\
\end{array}
\right.
\end{eqnarray*}
\end{Proposition}

\begin{proof}
Since $\psi:X\to P$ is an affine morphism, the Leray
spectral sequence degenerates so that we have 
$H^2(X, {\cal Z}^n)
= H^2(P, \psi_*(\OO_X(n\tilde{E})\tensor(\psi^*\circ\pi^*){\N}^n))
= H^2(P, \psi_*\OO_X(n\tilde{E})\tensor\pi^*{\N}^n)$.
Then by Lemma~\ref{tool2} and Lemma~\ref{tool1}, we compute
\begin{eqnarray*}
\lefteqn{H^2(X, {\cal Z}^n)}\\
& =& 
\left\{
\begin{array}{ll}
H^2(P, \OO_P(k)\tensor \pi^*{\N}^n)  & \\
\quad \dirsum\; \Dirsum_{i=1}^{\ell -1}
    H^2(P, 
       \OO_P\left(
               -\frac{i(p+1)}{\ell} + k
             \right)
        \tensor \pi^*{\N}^{ip +n}
      )
     & \mbox{if } n = k\ell \geq 0\\
 & \\
H^2(P, \OO_P(r+1+k-\ell))\tensor\pi^*{\N}^n)  &  \\
\quad\dirsum\;
\Dirsum_{i=1}^{\ell -1}
   H^2(P, 
       \OO_P\left(
                 -\frac{i(p+1)}{\ell} + k + 1
           \right)
       \tensor\pi^*{\N}^{ip+n}
       )
     & \mbox{if } n = k\ell + r > 0\\
 & \\
H^2(P, \OO_P(n)\tensor\pi^*{\N}^n) & \\
\quad\dirsum\; \Dirsum_{i=1}^{\ell-1}
H^2(P, 
     \OO_P\left(
             -\frac{i(p+1)}{\ell}
         \right)
     \tensor\pi^*{\N}^{ip+n}
    )
     & \mbox{if } n <0. \\
\end{array}
\right.
\end{eqnarray*}
Moreover, 
in the case of $n= k\ell\geq 0$,
we have  $-\frac{i(p+1)}{\ell}+k\geq 0$ if $i\leq \frac{n}{p+1}$.
Also in the case of $n=k\ell +r>0$,
we have $-\frac{i(p+1)}{\ell}+k + 1\geq 0$ if 
$i\leq\frac{n+\ell-r}{p+1}$.
Now by Lemma~\ref{prop:basicsRuledSurf}(ii), 
we have $H^2(P, \OO_P(j)\tensor \pi^*{\cal N}^n)  
= H^2(C, S^j({\cal E})\tensor {\cal N}^n) = 0$
for $j\geq 0$.  
Thus we do not have to consider 
the direct summands with the indices $i$ in the above specified 
ranges.
\end{proof}

As an immediate consequence, we have the following vanishing 
result.

\begin{Corollary}
\label{h2vanishing}
We have $H^2(X, {\cal Z}^n)=0$ 
$(i)$ if $\ell\mid n$ and $n\geq (\ell-1)(p+1)$, 
in particular $n \geq p(p+1)$, 
or
$(ii)$ if $\ell\nondiv{n}$ and $n\geq p(p+1)-1$.
In particular, 
$H^2(X, {\cal Z}^n)=0$ for all $n\geq p(p+1)$.
\end{Corollary}
\begin{proof}By Proposition~\ref{h2}
and Lemma~\ref{prop:basicsRuledSurf}$(ii)$,
we know that, for $n\geq 0$, $H^2(X, {\cal Z}^n)=0$
if $(i)$ $n=k\ell\geq 0$ and $\ell -1 < \left[\frac{n}{p+1}+1\right]$,
 or
$(ii)$ $n = k\ell + r>0$, $\ell -1 < \left[\frac{n+\ell -r}{p+1}+1\right]$
and $r+1+k - \ell \geq 0$. 
The second condition of $(i)$
is equivalent to
$\frac{n}{p+1}+1 - (\ell-1)\geq 1$, i.e.,
$n\geq (p+1)(\ell-1)$.  
Since $\ell\mid(p+1)$, we have $\ell-1\leq p$
so that $n\geq p(p+1)$ implies in particular $n\geq (p+1)(\ell-1)$. 
Similarly, the second condition of $(ii)$
is equivalent to $n\geq (p+1)(\ell-1)- (\ell - r)$.
Since $\ell\mid(p+1)$ and $1\leq r\leq \ell-1$, 
We have $(p+1)(\ell-1)- (\ell - r)
\leq p(p+1) - 1$. Thus in particular the second
condition of $(ii)$ is satisfied for $n\geq p(p+1)-1$.
Moreover, by the first and the third condition of $(ii)$, we have 
$n \geq (\ell-r)(\ell-1)$. Since $1\leq r$ and 
$\ell\mid(p+1)$, we have 
$(\ell-r)(\ell-1)\leq (\ell-1)^2 \leq p^2 < p(p+1)-1$.
Thus, in this case, $n\geq p(p+1)-1$ suffices for the vanishing.
\end{proof}

\subsection{computation of $H^1(X, {\cal Z}^n)$}

The case of $n\geq 0$  can be computed by the same method as 
in Proposition~\ref{h2} except the difference in 
the dimension of cohomologies.

\begin{Proposition}
\label{h1pos}
Let $n\geq 0$. For 
$k\geq 0$, $1\leq r\leq \ell-1$, we have 
\begin{eqnarray*}
\lefteqn{H^1(X, {\cal Z}^n)}\\
& =& 
\left\{
\begin{array}{ll}
H^1(C, S^k({\cal E})\tensor {\N}^n)  & \\
\quad \dirsum\; \Dirsum_{i=1}^{\left[\frac{n}{p+1}\right]}
    H^1(C, 
         S^{-\frac{i(p+1)}{\ell} + k}({\cal E})
        \tensor {\N}^{ip+n}
      )  & \\
\quad \dirsum\; \Dirsum_{i=\left[\frac{n}{p+1}+1\right]}^{\ell -1}
    H^1(P, 
       \OO_P\left(
               -\frac{i(p+1)}{\ell} + k
             \right)
        \tensor \pi^*{\N}^{ip+n}
      )
     & \mbox{if } n = k\ell \geq 0\\
 & \\
H^1(P, \OO_P(r+1+k-\ell))\tensor\pi^*{\N}^n)  &  \\
\quad\dirsum\;
\Dirsum_{i=1}^{\left[\frac{n+\ell -r}{p+1}\right]}
   H^1(C, 
       S^{-\frac{i(p+1)}{\ell} + k + 1}({\cal E})
       \tensor {\N}^{ip+n}
       )  & \\
\quad\dirsum\;
\Dirsum_{i=\left[\frac{n+\ell -r}{p+1} + 1\right]}^{\ell -1}
   H^1(P, 
       \OO_P\left(
                 -\frac{i(p+1)}{\ell} + k + 1
           \right)
       \tensor\pi^*{\N}^{ip+n}
       )

     & \mbox{if } n = k\ell + r > 0\\
\end{array}
\right.
\end{eqnarray*}
Moreover, by Lemma~\ref{prop:basicsRuledSurf}$(ii)$, 
the first term in the case of $n=k\ell + r>0$
is $H^1(P, \OO_P(r+1+k-\ell))\tensor\pi^*{\N}^n) 
\iso 
H^1(C, S^{r+1+k-\ell}({\cal E})\tensor{\N}^n)$
if $r+k\geq \ell -1$.
\end{Proposition}

Now we consider the case of $n<0$.

\begin{Theorem}
\label{h1neg}
For $n<0$, we have 
\begin{equation*}
H^1(X, {\cal Z}^n)
= 
  \Dirsum_{i=1}^{\ell-1} 
       H^0(C, S^{\frac{i(p+1)}{\ell}-2}({\cal E})^\vee
              \tensor{\N}^{ip-\ell +n}).
\end{equation*}
\end{Theorem}
\begin{proof}
Consider  a part of the five-term exact sequence 
\begin{equation*}
0 \To H^1(C, \phi_*{\cal Z}^n) \To H^1(X, {\cal Z}^n)
  \To H^0(C, R^1\phi_*{\cal Z}^n)
  \To H^2(C, \phi_*{\cal Z}^n)
\end{equation*}
for the Leray spectral sequence
$E_2^{p,q}= H^p(C, R^{q}\phi_*{\cal Z}^n) 
\Rightarrow H^{p+q}(X, {\cal Z}^n)$.
We have $H^2(C, \phi_*{\cal Z}^n)=0$ since $\dim C=1$,
and moreover an easy calculation using
Lemma~\ref{tool2} shows 
\begin{equation*}
H^1(C, \phi_*{\cal Z}^n)
= H^1(C, \pi_*\OO_P(n)\tensor{\N}^n)
    \dirsum
   \Dirsum_{i=1}^{\ell-1}
         H^1(C, \pi_*
		  \OO_P(-\frac{i(p+1)}{\ell})
                \tensor{\N}^{ip+n})
\end{equation*}
and this is $=0$ by Lemma~\ref{prop:basicsRuledSurf}(i).
Thus, we have 
\begin{equation}
\label{h1-eq:1}
H^1(X, {\cal Z}^n)=H^0(C, R^1\phi_*{\cal Z}^n)\quad(n<0).
\end{equation}
On the other hand, in a part of the five-term exact sequence 
\begin{equation*}
  0\To R^1\pi_*(\psi_*{\cal Z}^n)
   \To R^1\phi_*{\cal Z}^n
   \To \pi_*(R^1\psi_*{\cal Z}^n) 
\end{equation*}
for $\psi: X\to P$ and $\pi: P \to C$, 
we have 
$R^1\psi_*{\cal Z}^n=0$ since $\psi$ is an affine morphism.
Thus we have 
\begin{equation}
\label{h1-eq:2}
R^1\phi_*{\cal Z}^n = R^1\pi_*(\psi_*{\cal Z}^n)\quad (n\in \ZZ).
\end{equation}
Now an easy calculation using Lemma~\ref{tool2} 
and Lemma~\ref{prop:basicsRuledSurf}(iii) shows
\begin{equation*}
R^1\pi_*(\psi_*{\cal Z}^n)
= 
\left\{
\begin{array}{ll}
\Dirsum_{i=1}^{\ell-1}R^1\pi_*{\cal M}^i\tensor{\N}^{-1}
& \mbox{if }n = -1 \\
S^{-n-2}({\cal E})^\vee\tensor {\N}^{n-\ell}
\dirsum
     \Dirsum_{i=1}^{\ell-1}R^1\pi_*{\cal M}^i\tensor{\N}^n
& \mbox{if }n\leq -2 \\
\end{array}
\right.
\end{equation*}
But by Proposition~\ref{prop:vanishingONcurves}
and $\deg {\N}>0$, we have 
\begin{equation*}
H^0(C,S^{-n-2}({\cal E})^\vee\tensor {\N}^{n-\ell})=0 \quad 
\mbox{for }n\leq -2.
\end{equation*}
Thus by (\ref{h1-eq:1}) and (\ref{h1-eq:2})
we have 
\begin{equation}
\label{h1-eq:3}
H^1(X, {\cal Z}^n)
= \Dirsum_{i=1}^{\ell-1}
H^0(C, R^1\pi_*{\cal M}^i\tensor{\N}^n).
\end{equation}
By relative Serre duality,  the well-known formula
$\omega_{P/C} = \OO_P(-2)\tensor\pi^*{\cal L}$
and  Lemma~\ref{prop:basicsRuledSurf}$(i)$, we compute
\begin{equation*}
R^1\pi_*{\cal M^i} \iso \pi_*({\cal M}^{-i}\tensor\omega_{P/C})^\vee
= S^{\frac{i(p+1)}{\ell}-2}({\cal E})^\vee 
        \tensor {\N}^{ip-\ell}.
\end{equation*}
and we obtain the above stated result.
\end{proof}

We give here some specific instances of Theorem~\ref{h1neg}.
\begin{Example}
\label{h1neg-detailed}
Let $n<0$. Then,
\begin{itemize}
\item if $\ell = p+1$:
\begin{equation*}
H^1(X, {\cal Z}^n)
= H^0(C, {\N}^{p-1+n})
   \dirsum \Dirsum_{i=3}^p 
              H^0(C, S^{i-2}({\cal E})^\vee\tensor{\N}^{ip-p-1-n})
\end{equation*}
where the first term vanishes for $n<-(p-1)$. In particular, 
if $p=2$ (and then $\ell=3$), we have $H^1(X, {\cal Z}^{n})=0$ for
$n\leq -2$ and moreover $H^1(X, {\cal Z}^{-1})\ne{0}$ since 
this is exactly the Raynaud's counter-example.
\item if $2\ell = p+1$:
\begin{equation*}
H^1(X, {\cal Z}^n)
= H^0(C, {\N}^{\frac{p-1}{2}+n})
   \dirsum \Dirsum_{i=2}^{\frac{p-1}{2}} 
              H^0(C, S^{2i-2}({\cal E})^\vee
                     \tensor{\N}^{ip-\frac{p+1}{2}+n})
\end{equation*}
where the first term vanishes for $n<-\frac{p-1}{2}$.
In particular, if $p=3$ (and then $\ell=2$), we have $H^1(X, {\cal Z}^n)=0$ 
for $n\leq -2$ and moreover $H^1(X, {\cal Z}^{-1})\ne{0}$ since 
this is exactly the Raynaud's counter-example.
\end{itemize}
\end{Example}

Now we show some non-vanishing results.

\begin{Theorem}
\label{result1}
$H^1(X, {\cal Z}^n)\ne{0}$ for every $n$ such that 
$-(\ell-\lceil\frac{2\ell}{p+1}\rceil)\leq n\leq -1$,
where $\lceil\cdots\rceil$ denotes the round up.
\end{Theorem}
\begin{proof}
Since ${\cal L}={\cal N}^e = {\N}^\ell$ is the surjective image 
of ${\cal E}$ (cf. (\ref{fundamentalExSeq3}),
we have the short exact sequence
\begin{equation*}
           S^{\frac{k}{\ell}}({\cal E}) \To {\N}^{k}\To 0
\end{equation*}
for any $k\in\NN$ such that $\ell\mid{k}$. Taking the dual
and tensoring by ${\N}^k$, we obtain
\begin{equation*}
   0\To  \OO_C \To S^{\frac{k}{\ell}}({\cal E})^\vee\tensor{\N}^k.
\end{equation*}
Then we have 
\begin{equation*}
    k = H^0(C,\OO_C) 
    \subset H^0(C, S^{\frac{k}{\ell}}({\cal E})^\vee\tensor{\N}^k).
\end{equation*}
Applying this result, we know that 
the term $H^0(C, S^{\frac{i(p+1)}{\ell}-2}({\cal E})^\vee
              \tensor{\N}^{ip-\ell+n})$, $1\leq i\leq \ell-1$,
in  Theorem~\ref{h1neg} is non-trivial if 
\begin{equation*}
\frac{i(p+1)}{\ell}-2 \geq 0\quad\mbox{and}\quad
\ell\left(\frac{i(p+1)}{\ell}-2\right) = ip-\ell+n,
\end{equation*}
namely $n=-(\ell-i)$, with $\lceil\frac{2\ell}{p+1}\rceil\leq i\leq \ell-1$.
\end{proof}

For small characteristics, the evaluation of 
the non-vanishing degrees in Theorem~\ref{h1neg} 
is best possible.

\begin{Corollary}
\label{h1-for-p=2,3}
If $p=2$ or $3$, we have $H^1(X, {\cal Z}^n)=0$ for every
$n<-(\ell-\lceil\frac{2\ell}{p+1}\rceil)$.
\end{Corollary}
\begin{proof}
Since $\ell\mid p+1$, we have only to consider the cases
$(p, \ell) = (2, 3), (3, 2), (3,4)$ and 
\begin{equation*}
-\left(\ell-\lceil\frac{2\ell}{p+1}\rceil\right)
=\left\{
  \begin{array}{ll}
    -1 & \mbox{if } (p,\ell)=(2,3)\\
    -1 & \mbox{if } (p,\ell)=(3,2)\\
    -2 & \mbox{if } (p,\ell)=(3,4).\\
  \end{array}
\right.
\end{equation*}
The first two cases are already shown in Example~\ref{h1neg-detailed}.
Then we assume $(p,\ell)=(3,4)$ in the following. We have
\begin{eqnarray*}
H^1(X, {\cal Z}^n)
&=& \Dirsum_{i=1}^{3} 
       H^0(C, S^{i-2}({\cal E})^\vee
              \tensor{\N}^{3i-4 +n})\\
&=& H^0(X, {\N}^{2+n}) 
\dirsum H^0(C, S^1({\cal E})^\vee \tensor {\N}^{5+n})
\quad (n<0)
\end{eqnarray*}
by Theorem~\ref{h1neg}. Since $\deg{\N}>0$, we have 
$H^0(X, {\N}^{2+n})=0$ for $n<-2$.
Moreover, since ${\N}^{5+n}= {\cal N}^{\frac{e(5+n)}{4}}$
by definition and since $\frac{e(5+n)}{4}<e$ for $n<-2$,
we have 
$H^0(C, S^1({\cal E})^\vee \tensor {\N}^{5+n})=0$ 
for $n<-2$ by Proposition~\ref{prop:vanishingONcurves}.
Thus, in this case we have $H^1(X, {\cal Z}^n)=0$ for $n<-2$.
\end{proof}

\subsection{computation of $H^0(X, {\cal Z}^n)$}

Since ${\cal Z}$ is ample, we have $H^0(X,{\cal Z}^n)=0$ for $n<0$.
For $n\geq 0$, we have

\begin{Proposition} 
\label{h0}
For $n, k\geq 0$ and  $1\leq r\leq \ell-1$, we have 
\begin{eqnarray*}
\lefteqn{H^0(X, {\cal Z}^n)}\\
&=& \left\{
\begin{array}{ll}
\Dirsum_{i=0}^{\ell-1}
  H^0(C, S^{-\frac{i(p+1)}{\ell}+k}({\cal E})
             \tensor {\N}^{ip+n})   
    & \mbox{if }n=k\ell\geq 0\\
  & \\
H^0(C, S^{r+1-k-\ell}({\cal E})\tensor {\N}^n) & \\
\dirsum 
\Dirsum_{i=1}^{\ell-1}
    H^0(C, S^{-\frac{i(p+1)}{\ell}+k+1}({\cal E})
           \tensor {\N}^{ip+n})
    & \mbox{if }n=k\ell+r> 0.\\
\end{array}
\right.
\end{eqnarray*}
\end{Proposition}
\begin{proof} By Lemma~\ref{tool1}, we compute
\begin{eqnarray*}
H^0(X, {\cal Z}^n)
&= & H^0(P, \psi_*\OO_X(n\tilde{E})\tensor\pi^*{\N}^n)\\
&=&
\left\{
\begin{array}{ll}
\Dirsum_{i=0}^{\ell-1}
  H^0(C, \pi_*{\cal M}^i(kE)\tensor {\N}^n)
    & \mbox{if }n=k\ell \geq 0\\
 & \\
H^0(C, \pi_*\OO_P(r+1+k-\ell)\tensor {\N}^n) & \\
\dirsum 
\Dirsum_{i=1}^{\ell-1}
    H^0(C, \pi_*{\cal M}^i((k+1)E)\tensor {\N}^n)
    & \mbox{if }n=k\ell+r> 0\\
\end{array}
\right.\\
\end{eqnarray*}
Then apply Lemma~\ref{prop:basicsRuledSurf}(i).
\end{proof}

\begin{Remark}{\em
According to Proposition~\ref{h0}, we know that 
the lower bound $B$ such 
that $H^0(X, {\cal Z}^n)=0$ for $n\geq B$,
depends on the vanishing of 
cohomologies of type 
$H^0(C, S^{m}({\cal E})\tensor {\N}^n)$ for $m,n\gg 0$.
Hence it seems to be difficult to give a general 
estimation of $B$.
}
\end{Remark}

By the similar argument as in the proofs of 
Proposition~\ref{h2} and~\ref{h1pos}, we know that 
we have only to consider fewer direct summands 
than Proposition~\ref{h0} in some cases. Namely, 
\begin{itemize}
\item if $n=k\ell$ and $0\leq n<(p+1)(\ell-1)$, we have 
\begin{equation*}
H^0(X, {\cal Z}^n)
= \Dirsum_{i=0}^{\left[\frac{n}{p+1}\right]}
  H^0(C, S^{-\frac{i(p+1)}{\ell}+k}({\cal E})
             \tensor {\N}^{ip+n})   
\end{equation*}
\item  if $n=k\ell + r$, $0<r\leq \ell-1$, and
$0< n <  p(\ell-1)+r-1$, we have 
\begin{eqnarray*}
\lefteqn{H^0(X, {\cal Z}^n)}\\
& =& 
H^0(C, S^{r+1+k-\ell}({\cal E})\tensor {\N}^n) 
\dirsum 
\Dirsum_{i=1}^{\left[\frac{n+\ell-r}{p+1}\right]}
    H^0(C, S^{-\frac{i(p+1)}{\ell}+k+1}({\cal E})
           \tensor {\N}^{ip+n})
\end{eqnarray*}
\end{itemize}
Then we have 

\begin{Corollary}
\label{h0-refined}
$H^0(X, {\cal Z}^n)=0$
if $n = k\ell +r>0$ with $0< r\leq \ell-1$
and $0\leq k\leq 
\min\{\ell-r-2,\; \frac{p+1}{\ell}-2\}$.
\end{Corollary}
\begin{proof}
By the above formula for $n=k\ell +r$, $0<r\leq \ell-1$,
we know that 
$H^0(X, {\cal Z}^n)=0$ if 
$r+1+k-\ell <0$ and $n+\ell -r < p+1$.
From the latter inequation, we have 
\begin{equation*}
n= k\ell +r < \min\{p+1 - \ell +r,\;  p(\ell-1)+r-1\}
\end{equation*}
and then together with the former inequation we have 
\begin{equation*}
    k < \min\left\{
               \ell -r -1,\;
               \frac{p+1}{\ell}-1, \;
               p - \frac{p+1}{\ell}
           \right\}
      =  \min\{\ell-r-1,\; \frac{p+1}{\ell}-1\}
\end{equation*}
where the last equation is by $\ell\geq 2$.

\end{proof}

\section{Families of non-vanishing polarizations}

We have considered the Mumford-Szpiro type polarization
given in Proposition~\ref{prop:polarization}. 
Raynaud's example is also of this kind. 
We show that much more varieties of polarizations can serve 
as counter-examples to Kodaira vanishing. We fist consider 
\begin{equation*}
{\cal Z}_{a,b} :=  \OO_X(a\tilde{E})\tensor\phi^*{\N}^b
\quad(a,b\geq 1).
\end{equation*}

\begin{Proposition}
${\cal Z}_{a,b}$ is ample.
\end{Proposition}
\begin{proof}
We have $E^2 = \deg D>0$ and also $E.C>0$ for every irreducible 
curve $C\in P$ (see Prop.~V.2.3 \cite{H}).
Thus $\OO_P(nE)$, $n>0$, is ample by Nakai-Moishezon criteria
and,  since $\psi:X\to P$ is a finite morphism, 
$\psi^*\OO_P(nE) = \OO_X(\ell n \tilde{E})$, $n>0$,
is also ample. 
In particular, $\OO_X(a\tilde{E})$, $a\geq 1$, 
is  ample. 
On the other hand, ${\N}^b = {\cal N}^{be/\ell}$, $b\geq 1$,
is ample so that in particular $\pi^*{\N}^b$
is semi-ample (i.e., its sufficiently large powers are generated by
global sections). 
Consequently, $\OO_X(a\tilde{E})\tensor\phi^*{\N}^{b}$ 
is ample. 
\end{proof}

Then, by carrying out a similar argument as 
the proofs of Theorem~\ref{h1neg} and Theorem~\ref{result1},
we have 
\begin{Theorem}
\label{prop:outkvce}
$H^1(X, {\cal Z}_{a,b}^{-1})\ne{0}$
for all $a\geq 1$ and $1\leq b\leq \ell -1$.
\end{Theorem}
\begin{proof}
Let ${\cal Q}$ be an invertible sheaf on $C$  and set 
${\cal Z}=\OO_X(a\tilde{E})\tensor\phi^*{\cal Q}$. Now consider 
the following Leray spectral sequence of $\phi:X\To C$
\begin{equation*}
E^{p,q}_2 
= H^p(C, R^{q}\phi_*{\cal Z}^{-1})
\Rightarrow H^{p+q}(X, {\cal Z}^{-1})
\quad(p\geq 0).
\end{equation*}
We have $E_2^{2,0}=0$ since $\dim C=1$. Then 
by the 5-term exact sequence we have 
\begin{equation*}
  H^1(X, {\cal Z}^{-1})
 \To H^0(C, R^1\phi_*{\cal Z}^{-1})
 \To 0.
\end{equation*}
Thus we have only to show 
$H^0(C, R^1\phi_*{\cal Z}^{-1})
=H^0(C, R^1\phi_*\OO_X(-a\tilde{E}) \tensor{\cal Q}^{-1})
\ne{0}$.
Considering the 5-term exact sequence 
\begin{equation*}
0\To R^1\pi_*(\psi_*\OO_X(-a\tilde{E}))
 \To R^1(\pi\circ\psi)_*\OO_X(-a\tilde{E})
 \To \pi^*(R^1\psi_*\OO_X(-a\tilde{E})),
\end{equation*}
where $R^1\psi_*\OO_X(-a\tilde{E})=0$ since $\psi: X\to P$
is an affine morphism, we have 
\begin{equation*}
 R^1\phi_*\OO_X(-a\tilde{E})
= R^1(\pi\circ\psi)_*\OO_X(-a\tilde{E})
 \iso
 R^1\pi_*(\psi_*\OO_X(-a\tilde{E})).
\end{equation*}
Thus by Lemma~\ref{tool2} we obtain
\begin{eqnarray*}
R^1\phi_*\OO_X(-a\tilde{E})
&=& R^1\pi_*\OO_P(-aE)\; \dirsum\; 
\Dirsum_{i=1}^{\ell-1}(R^1\pi_*{\cal M}^i)\\
&=& 
R^1\pi_*\OO_P(-aE)\; \dirsum\; 
\Dirsum_{i=1}^{\ell-1}(S^{(ip+i-2\ell)/\ell}({\cal E})
       \tensor {\cal N}^{(\ell-ip)e/\ell}
          )^\vee.
\end{eqnarray*}
We note that the last equation is shown
in the end of the proof of Theorem~\ref{h1neg}.
Thus  we have 
\begin{equation*}
H^0(C, R^1\phi_*\OO_X(-a\tilde{E})\tensor {\cal Q}^{-1})
\supseteq \Dirsum_{i=1}^{\ell-1}
      H^0(C, (S^{(ip+i-2\ell)/\ell}({\cal E})
       \tensor {\cal N}^{(\ell-ip)e/\ell}
          )^\vee \tensor {\cal Q}^{-1}).
\end{equation*}
(Actually we can show that this inclusion is really an equation.)

On the other hand, from ${\cal E}\To {\cal L}\To 0$ we have 
\begin{equation*}
S^{(ip+i-2\ell)/\ell}({\cal E})
\tensor {\cal N}^{(\ell-ip)e/\ell}
\To 
{\cal L}^{(ip+i-2\ell)/\ell}
\tensor {\cal N}^{(\ell-ip)e/\ell}
\To 0.
\end{equation*}
Taking the dual and tensoring by ${\cal Q}^{-1}$, we have 
\begin{equation*}
0\To 
({\cal L}^{(ip+i-2\ell)/\ell}
\tensor {\cal N}^{(\ell-ip)e/\ell})^\vee\tensor {\cal Q}^{-1}
\To 
(S^{(ip+i-2\ell)/\ell}({\cal E})
\tensor {\cal N}^{(\ell-ip)e/\ell})^\vee\tensor {\cal Q}^{-1}
\end{equation*}
and 
\begin{equation*}
({\cal L}^{(ip+i-2\ell)/\ell}
\tensor {\cal N}^{(\ell-ip)e/\ell})^\vee
= ({\cal N}^{(ip+i-2\ell)e/\ell}
\tensor {\cal N}^{(\ell-ip)e/\ell})^\vee
= {\cal N}^{(\ell-i)e/\ell}.
\end{equation*}
Thus we have  for $i=1,\ldots, \ell-1$
\begin{equation*}
H^0(C,{\cal N}^{(\ell-i)e/\ell}\tensor {\cal Q}^{-1})
\subset H^0(C, R^1\phi_*\OO_X(-a\tilde{E})\tensor {\cal Q}^{-1}).
\end{equation*}
In particular, 
taking ${\cal Q}= {\cal N}^{\frac{(\ell-i)e}{\ell}}
= {\N}^{\ell-i}$ with $i=1,\ldots, \ell-1$, 
we have
$k 
\subset 
H^0(C, R^1\phi_*\OO_X(-a\tilde{E})\tensor {\cal Q}^{-1})$
as required and in this case 
${\cal Z}$ is exactly what we defined.
\end{proof}

The cohomologies 
$H^i(X, {\cal Z}_{a,b}^{n})$, $i,n\in\ZZ$, can also be computed 
by a similar method to what we have described.



\end{document}